\documentclass[11pt]{article}
\usepackage[pagebackref=true,colorlinks,linkcolor=blue,citecolor=magenta]{hyperref}
\usepackage{amssymb,amsmath}
\usepackage{amsmath}
\usepackage{verbatim}
\usepackage{dsfont}
\usepackage{amsfonts}
\newtheorem{precor}{{\bf Corollary}}

\newtheorem{precon}{{\bf Conjecture}}

\newenvironment{con}{\begin{precon}{\hspace{-0.5
               em}{\bf.\ }}}{\end{precon}}
\newtheorem{prealphcon}{{\bf Conjecture}}

\newtheorem{predefin}{{\bf Definition}}

\newenvironment{defin}[1]{\begin{predefin}{\hspace{-0.5
                   em}{\bf.\ }}{\rm #1}\hfill{$\blacktriangle$}}{\end{predefin}}
\newtheorem{preexm}{{\bf Example}}

\newenvironment{exm}[1]{\begin{preexm}{\hspace{-0.5
                  em}{\bf.\ }}{\rm #1}\hfill{$\bigstar$}}{\end{preexm}}

\newtheorem{preappl}{{\bf Application}}

\newtheorem{prelem}{{\bf Lemma}}

\newenvironment{lem}{\begin{prelem}{\hspace{-0.5
               em}{\bf.\ }}}{\end{prelem}}
\newtheorem{preproof}{{\bf Proof.\ }}

\newenvironment{proof}[1]{\begin{preproof}{\rm
               #1}\hfill{$\blacksquare$}}{\end{preproof}}
\newtheorem{prethm}{{\bf Theorem}}

\newenvironment{thm}{\begin{prethm}{\hspace{-0.5
               em}{\bf.\ }}}{\end{prethm}}
\newtheorem{prealphthm}{{\bf Theorem}}

\newtheorem{prealphlem}{{\bf Lemma}}

\newtheorem{prepro}{{\bf Proposition}}

\newtheorem{preprb}{{\bf Problem}}

\newenvironment{prb}{\begin{preprb}{\hspace{-0.5
               em}{\bf.\ }}}{\end{preprb}}
\newtheorem{prerem}{{\bf Remark}}

\newenvironment{rem}{\begin{prerem}{\hspace{-0.5
               em}{\bf.\ }}}{\end{prerem}}
\newtheorem{preapp}{{\bf Application}}

\newtheorem{prequ}{{\bf Question}}

\newtheorem{preclaim}{{\bf Claim}}

\def\conct[#1,#2]{\mbox {${#1} \leftrightarrow {#2}$}}
\def\dconct[#1,#2]{\mbox {${#1} \rightarrow {#2}$}}
\def\deg[#1,#2]{\mbox {$d_{_{#1}}(#2)$}}
\def\mindeg[#1]{\mbox {$\delta_{_{#1}}$}}
\def\maxdeg[#1]{\mbox {$\Delta_{_{#1}}$}}
\def\outdeg[#1,#2]{\mbox {$d_{_{#1}}^{^+}(#2)$}}
\def\minoutdeg[#1]{\mbox {$\delta_{_{#1}}^{^+}$}}
\def\maxoutdeg[#1]{\mbox {$\Delta_{_{#1}}^{^+}$}}
\def\indeg[#1,#2]{\mbox {$d_{_{#1}}^{^-}(#2)$}}
\def\minindeg[#1]{\mbox {$\delta_{_{#1}}^{^-}$}}
\def\maxindeg[#1]{\mbox {$\Delta_{_{#1}}^{^-}$}}

\def\dre[#1,#2,#3]{\mbox {${\cal E}^{^{#3}}(#1,#2)$}}
\def\var[#1,#2]{\mbox {${\rm Var}_{_{#1}}(#2)$}}
\def\ls[#1]{\mbox {$\xi^{^{#1}}$}}
\def\hom[#1,#2]{\mbox {${\rm Hom}({#1},{#2})$}}
\def\onvhom[#1,#2]{\mbox {${\rm Hom^{v}}(#1,#2)$}}
\def\onehom[#1,#2]{\mbox {${\rm Hom^{e}}(#1,#2)$}}
\def\core[#1]{\mbox {$#1^{^{\bullet}}$}}
\def\cay[#1,#2]{\mbox {${\rm Cay}({#1},{#2})$}}
\def\sch[#1,#2,#3]{\mbox {${\rm Sch}({#1},{#2},{#3})$}}
\def\cays[#1,#2]{\mbox {${\rm Cay_{s}}({#1},{#2})$}}
\def\dirc[#1]{\mbox {$\stackrel{\rightarrow}{C}_{_{#1}}$}}
\def\cycl[#1]{\mbox {${\bf Z}_{_{#1}}$}}

\begin{document}

\begin{center}
{\Large \bf Concerning Two Conjectures of Frick and Jafari}\\
\vspace{0.3 cm}
{\bf Saeed Shaebani\\
{\it School of Mathematics and Computer Science}\\
{\it Damghan University}\\
{\it P.O. Box {\rm 36716-41167}, Damghan, Iran}\\
{\tt shaebani@du.ac.ir}\\
}
\end{center}
\begin{abstract}
\noindent In 2022, Hamid Reza Daneshpajouh provided some counterexamples to the following conjecture of Florian Frick.\\

\noindent {\bf Conjecture.}
	Let $r \geq 3$. Then, every hypergraph ${\cal G}$ over the
	ground set $[n]$ satisfies
	$$ \chi \left({\rm KG}^r ({\cal G}_{r-{\rm stable}})\right) \geq
	\left\lceil  \frac{{\rm cd}^r ({\cal G})}{r-1}   \right\rceil . $$

\noindent In this paper, we improve Danashpajouh's results. Also, we provide several counterexamples to the following recent conjecture of Amir Jafari.\\

\noindent {\bf Conjecture.}
	If $s \geq r \geq 2$,
	then for any hypergraph ${\cal G}$ with vertex set $[n]$ we have
	$$ \chi \left({\rm KG}^r ({\cal G}_{s-{\rm stable}})\right) \geq
	\left\lceil  \frac{{\rm ecd}^s ({\cal G})}{r-1}   \right\rceil . $$\\

\end{abstract}
\section{Introduction}

In 2020, Frick \cite{Frick2020} proposed the following conjecture.

\begin{con} \cite{Frick2020} \label{FrickConjecture}
	Let $\{n,r\} \subseteq \{3,4,5, \dots \}$ and ${\cal G}$ be a hypergraph over the
	ground set $[n]$. Then, we have
	$$ \chi \left({\rm KG}^r ({\cal G}_{r-{\rm stable}})\right) \geq
	\left\lceil  \frac{{\rm cd}^r ({\cal G})}{r-1}   \right\rceil . $$
\end{con}

\begin{defin}{
		Let ${\cal H}$ be a hypergraph and $r \in \{2,3, \dots\}$.
		{\it The equitable $r$-colorability defect} of ${\cal H}$, denoted by ${\rm ecd}^{r} ({\cal H})$,
		is the least size of a subset $S$ of $V({\cal H})$ that the induced subhypergraph of ${\cal H}$
		on $V({\cal H}) \setminus S$ is equitably $r$-colorable.
		In other words, ${\rm ecd}^{r} ({\cal H})$ is the minimum size of a subset $S$ of $V({\cal H})$ for which $V({\cal H}) \setminus S$
		admits a partition $\{X_{1} , X_{2} , \dots , X_{r}\}$ in such a way that the following two conditions are
		fulfilled :
		\begin{itemize}
			\item For any two distinct elements $i$ and $j$ in $[r]$, we have $ \left|   \left| X_{i} \right|  -   \left| X_{j} \right|     \right|  \leq 1$.
			\item For each hyperedge $e$ in $E({\cal H})$ and each $i$ in $[r]$, we have $e \nsubseteq X_{i}$.		
		\end{itemize}
	}		
\end{defin}

\noindent We always have ${\rm ecd}^{r} ({\cal H}) \geq {\rm cd}^{r} ({\cal H})$. The equitable $r$-colorability defect of hypergraphs was introduced by
Abyazi Sani and Alishahi in order to obtain a stronger sharp lower bound for
$\chi \left(   {\rm KG}^{r} ({\cal H})    \right)$.
\begin{thm} \cite{AbAl}
	For any hypergraph ${\cal H}$ and any positive integer $r$ with $r\geq 2$ we have
	$$\chi \left(   {\rm KG}^{r} ({\cal H})    \right) \geq
	\left\lceil    \frac{{\rm ecd}^{r} ({\cal H})}{r-1}    \right\rceil .$$
\end{thm}

The following conjecture was proposed by Jafari in 2022 \cite{Jafari2022}.

\begin{con} \cite{Jafari2022} \label{JafariConjecture}
	Let $\{n , r , s \} \subseteq \{2,3, \dots\}$ and $s \geq r$.
	Then, for any hypergraph ${\cal G}$ whose vertex set is $[n]$ we have
	$$ \chi \left({\rm KG}^r ({\cal G}_{s-{\rm stable}})\right) \geq
	\left\lceil  \frac{{\rm ecd}^s ({\cal G})}{r-1}   \right\rceil . $$
\end{con}

\noindent
The conjectures \ref{FrickConjecture} and \ref{JafariConjecture} could have a trivial counterexample, independent of the condition $s\geq r$.

\begin{exm}{ \label{Trivial}
		Let $\{r,s\} \subseteq \{2,3,4, \dots\}$. Put $n:=rs-1$ and consider a hypergraph
		${\cal G}$ with vertex set $[n]$ and hyperedge set $E({\cal G}) := {[n] \choose r}$.
		Now, since $E \left({\cal G}_{s-{\rm stable}}\right) = \varnothing$,
		it follows that
		$ \chi \left({\rm KG}^r ({\cal G}_{s-{\rm stable}})\right) = 0 $, while
		$$\left\lceil  \frac{{\rm ecd}^s ({\cal G})}{r-1}   \right\rceil = \left\lceil  \frac{{\rm cd}^s ({\cal G})}{r-1}   \right\rceil =
		\left\lceil  \frac{s-1}{r-1}   \right\rceil  \geq 1.$$
		This shows that the conjecture could not even be true if we replace a lower value
		$\left\lceil  \frac{{\rm cd}^s ({\cal G})}{r-1}   \right\rceil $
		instead of $\left\lceil  \frac{{\rm ecd}^s ({\cal G})}{r-1}   \right\rceil $.
		
		\noindent
		Moreover, if we regard an arbitrary $r$ as fixed and let $s$ vary in order to change our hypergraph ${\cal G}$ to a sequence ${\cal G}^{s}$ of hypergraphs depending on $s$,
		we obtain
		$$\chi ({\rm KG}^r ({\cal G}^{s}_{s-{\rm stable}}))=0 $$
		while
		$$\lim _{s \to \infty}\left\lceil  \frac{{\rm cd}^s ({\cal G}^{s})}{r-1}   \right\rceil
		=
		\lim _{s \to \infty}\left\lceil  \frac{{\rm ecd}^s ({\cal G}^{s})}{r-1}   \right\rceil
		=
		\infty .
		$$
	}
\end{exm}

\begin{prb} \label{GeneralProblem1}
	In Example \ref{Trivial}, the number of vertices of the hypergraph ${\cal G}$ is less than or equal to $rs-1$. For arbitrary but fixed positive integers $r, s,$ and $l$, does there exist another counterexample whose number of vertices is greater than $\max \{2r , 2s , l\}$?
\end{prb}

\begin{prb} \label{GeneralProblem2}
	In Example \ref{Trivial}, we have
	$E({\cal G}_{s-{\rm stable}}) = \varnothing$. For arbitrary but fixed $r$ and $s$, does there exist another counterexample which satisfies $E({\cal G}_{s-{\rm stable}}) \neq \varnothing$? And if so, could the number of vertices be arbitrarily large? or it must be bounded by a function of $r$ and $s$?  
\end{prb}

\begin{prb} \label{GeneralProblem3}
	In Example \ref{Trivial}, for the case where $r \geq s$, we have
	$$\left\lceil  \frac{{\rm cd}^r ({\cal G})}{r-1}   \right\rceil      -  \chi ({\rm KG}^r ({\cal G}_{s-{\rm stable}}))  =1.$$
	Could
	$\left\lceil  \frac{{\rm cd}^r ({\cal G})}{r-1}   \right\rceil      -  \chi ({\rm KG}^r ({\cal G}_{s-{\rm stable}}))$ be arbitrarily large?
\end{prb}

In this paper, we investigate The Problems \ref{GeneralProblem1},
\ref{GeneralProblem2}, and \ref{GeneralProblem3}.
In this regard, we prove the following results.

\subsection{For $r \geq s$}

\begin{thm} \label{Forrgeqs1}
	Let $n \in \mathbb{N}$ and $\{r,s\} \subseteq \{2,3,4,\dots \}$ with
	$r\geq s$.
	Then, there are not any hypergraphs ${\cal H}$ over the ground set $[n]$
	in such a way that
	$$\begin{array}{ccc}
		\chi ({\rm KG}^r ({\cal H}_{s-{\rm stable}}))=0    &   {\rm and}    &    \left\lceil  \frac{{\rm cd}^r ({\cal H})}{r-1}   \right\rceil  \geq 2.        \\
	\end{array}$$
\end{thm}

\begin{thm} \label{directrs}
	Let $s \in \{4,5,6, \dots \}$ and 
	$r \in \left\{s+1, s+2, \dots , s+ \left\lceil  \frac{s-3}{2}   \right\rceil \right\}.$
	Then, there exist an integer $n$ in $\{2r+1, 2r+2 , \dots \}$ and a graph ${\cal G}$ over the ground set $[n]$ such that
	$$\begin{array}{ccc}
		\chi ({\rm KG}^r ({\cal G}_{s-{\rm stable}})) = 0  &  {\rm and}    &    
		\left\lceil  \frac{{\rm cd}^r ({\cal G})}{r-1}   \right\rceil =1 .  \\
	\end{array}$$
\end{thm}

\begin{thm} \label{inversers}
	Let $\{n,r,s\} \subseteq \{4,5,6,\dots \}$ with
	$n\geq 2r$ and $r\geq s+1$.
	If there exists a hypergraph ${\cal H}$ over the ground set $[n]$
	in such a way that
	$$\begin{array}{ccc}
		\chi ({\rm KG}^r ({\cal H}_{s-{\rm stable}}))=0    &   {\rm and}    &    \left\lceil  \frac{{\rm cd}^r ({\cal H})}{r-1}   \right\rceil  = 1,        \\
	\end{array}$$
	then
	$$r \in \left\{s+1, s+2, \dots , s+ \left\lceil  \frac{s-3}{2}   \right\rceil \right\}.$$
\end{thm}

In order to provide a partial answer to The Problem \ref{GeneralProblem3}, first we
propose the following conjecture.

\begin{con} \label{Almost}
	Let $\{n, r, s\} \subseteq \{2, 3, 4, \dots \}$ and $r \geq s$.
	Then, for any hypergraph ${\cal G}$ over the ground set $[n]$ we have
	$$  \chi \left(       {\rm KG}^r \left( {\cal G}_{     \widetilde{{s-  {\rm stable}   }}     } \right)         \right)
	\geq
	\left\lceil  \frac{{\rm cd}^r ({\cal G})}{r-1}   \right\rceil    . $$
\end{con}

Now, the correctness of The Conjecture \ref{Almost} would completely solve The Problem \ref{GeneralProblem3}, as follows. 

\begin{thm} \label{AlmosttoProblem3}
	If the Conjecture \ref{Almost} is true, then The Problem \ref{GeneralProblem3}
	would have a negative answer; and moreover, we would always have
	$$\left\lceil  \frac{{\rm cd}^r ({\cal G})}{r-1}   \right\rceil      -  \chi ({\rm KG}^r ({\cal G}_{s-{\rm stable}})) \leq 1.$$
\end{thm}

\subsection{For Free $r$ and $s$}

\begin{thm} \label{Freers1}
	Let $\{r,s, l\} \subseteq \{2,3,4,5,\dots \}$.
	Then, there exist an integer $n$ larger than $l$ and a graph ${\cal G}$
	over the ground set $[n]$ in such a way that
	$$\begin{array}{ccc}
		\chi ({\rm KG}^r ({\cal G}_{s-{\rm stable}}))=0  &  {\rm and}   &
		\left\lceil  \frac{{\rm cd}^s ({\cal G})}{r-1}   \right\rceil  \geq 1.
	\end{array}$$
\end{thm}

\begin{thm} \label{Freers2}
	Let $\{r,s,l\} \subseteq \{2,3,4,5, \dots \}$.
	Then, there exist an integer $n$ larger than $l$ and a hypergraph ${\cal F}$ with vertex set $[n]$ for which
	$$\begin{array}[1]{ccc}
		\chi ({\rm KG}^r ({\cal F}_{s-{\rm stable}}))=1  &  {\rm and}   &
		\left\lceil  \frac{{\rm cd}^s ({\cal F})}{r-1}   \right\rceil  \geq 2.
	\end{array}  $$
\end{thm}


\section{Proofs of The Main Results}

\begin{defin}{
	Let $n$ and $s$ be positive integers. We define the graph
	${\cal F}_{n}^{s}$ as a graph with vertex set $[n]$ whose
	edge set is $E \left(  {\cal F}_{n}^{s}  \right) := {[n] \choose 2} \setminus {[n] \choose 2}_{s-{\rm stable}}$.
}
\end{defin}

\begin{lem} \label{restriction}
	Let $r, s, k,$ and $n$ be positive integers with $r\geq 2$.
	Then, there exists a hypergraph ${\cal H}$ over the ground set $[n]$ with
	$$\begin{array}[1]{ccc}
		\chi ({\rm KG}^r ({\cal H}_{s-{\rm stable}}))=0  &  {\rm and}   &
		\left\lceil  \frac{{\rm cd}^r ({\cal H})}{r-1}   \right\rceil  \geq k,
	\end{array}  $$
	if and only if the graph ${\cal F}_{n}^{s}$ satisfies
	$$\begin{array}[1]{ccc}
		\chi ({\rm KG}^r (\left({\cal F}_{n}^{s}\right)_{s-{\rm stable}}))=0  &  {\rm and}   &
		\left\lceil  \frac{{\rm cd}^r \left({\cal F}_{n}^{s}\right)}{r-1}   \right\rceil  \geq k.
	\end{array}  $$	
\end{lem}

\begin{proof}{
		We just prove the "only if" part of the Lemma.
		
		\noindent Since $\chi ({\rm KG}^r ({\cal H}_{s-{\rm stable}}))=0$,
		the hypergraph ${\cal H}$ does not have any $s$-stable hyperedges. In particular,
		${\cal H}$ does not have any singleton hyperedges, and every hyperedge of ${\cal H}$
		contains at least two vertices. Now, the hypergraph
		$\tilde{{\cal H}} := {\cal H} \cup {\cal F}_{n}^{s}$, which is a hypergraph with vertex set $[n]$ and hyperedge set
		$E \left(  \tilde{{\cal H}}   \right)  :=   E({\cal H})  \cup  E({\cal F}_{n}^{s})$, does not have any $s$-stable hyperedges, and satisfies
		$$\begin{array}[1]{ccc}
			\chi \left({\rm KG}^r \left( \left(  \tilde{{\cal H}}   \right)_{s-{\rm stable}}\right)\right)=0  &  {\rm and}   &
			\left\lceil  \frac{{\rm cd}^r \left(  \tilde{{\cal H}}   \right)}{r-1}   \right\rceil  \geq \left\lceil  \frac{{\rm cd}^r \left(  {\cal H}   \right)}{r-1}   \right\rceil  \geq k.
		\end{array}  $$
		We have $E({\cal F}_{n}^{s})  \subseteq E \left(  \tilde{{\cal H}}   \right) $. Also,
		for each hyperedge $\varepsilon$ in $E \left(  \tilde{{\cal H}}   \right)$, there exists an
		edge $e$ in $E({\cal F}_{n}^{s})$ for which $e \subseteq \varepsilon$.
		Consequently,
		$$\chi ({\rm KG}^r (\left({\cal F}_{n}^{s}\right)_{s-{\rm stable}}))
		\leq
		\chi \left({\rm KG}^r \left( \left(  \tilde{{\cal H}}   \right)_{s-{\rm stable}}\right)\right)=0,$$
		and also,
		$$\left\lceil  \frac{{\rm cd}^r \left({\cal F}_{n}^{s}\right)}{r-1}   \right\rceil
		=
		\left\lceil  \frac{{\rm cd}^r \left(  \tilde{{\cal H}}   \right)}{r-1}   \right\rceil  \geq k;
		$$
		and the proof is complete.
	}
\end{proof}

\begin{rem}
	It should be noted that in Lemma \ref{restriction}, one could drop the condition $\chi ({\rm KG}^r (\left({\cal F}_{n}^{s}\right)_{s-{\rm stable}})) = 0$. Because, we always have
	$\chi ({\rm KG}^r (\left({\cal F}_{n}^{s}\right)_{s-{\rm stable}}))=0$. Nevertheless, it is appeared in Lemma \ref{restriction}, just in order to indicate that whenever $n, r, s,$ and $k$ are fixed, then for deciding whether there exist some hypergraphs ${\cal H}$ over the ground set $[n]$ for which $\chi ({\rm KG}^r ({\cal H}_{s-{\rm stable}}))=0$ and $\left\lceil  \frac{{\rm cd}^r ({\cal H})}{r-1}   \right\rceil  \geq k$, it is sufficient to restrict our attention just to the graph ${\cal F}_{n}^{s}$.
\end{rem}

\begin{thm} \label{ChromaticNumber}
	Let $\{n,a,s\} \subseteq \{2,3,4, \dots\}$ and $b$ be a nonnegative integer such that
	$n=as+b$ and $0 \leq b \leq s-1$. Then, we have
	$\chi \left(  {\cal F}_{n}^{s}  \right) = \left\lceil   \frac{n}{a}   \right\rceil$. 
\end{thm}

\begin{proof}{
		Let $I$ be a nonempty independent set in $ {\cal F}_{n}^{s} $.
		For any two vertices $x$ and $z$ in $I$ with $x < z$, there exist at least $2s-2$ pairwise distinct
		vertices $y_{1} , y_{2} , \dots , y_{2s-2}$ in $[n] \setminus I$ for which
		$$\{y_{1} , y_{2} , \dots , y_{s-1}\}  \subseteq \{i \in {\mathbb{Z}} : x < i < z\},$$
		and also,
		$$\{y_{s-1} , y_{s-2} , \dots , y_{2s-2}\} \cap \{i \in {\mathbb{Z}} : x < i < z\} = \varnothing .$$
		Therefore,
		$$|I| + (s-1) |I| \leq n=as+b < as+s.$$
		Hence, $|I| \leq a$. Consequently, the independence number of $ {\cal F}_{n}^{s} $,
		denoted by $\alpha \left(  {\cal F}_{n}^{s}  \right)$, is less than or equal to $a$.
		Now, since the set $\{s, 2s, 3s, \dots , as\}$ represents an independent
		set of size $a$ in $ {\cal F}_{n}^{s} $, we obtain $\alpha \left(  {\cal F}_{n}^{s}  \right) = a$.
		Now, because of
		$\chi \left(  {\cal F}_{n}^{s}  \right)  \geq \left\lceil  \frac{n}{\alpha \left(  {\cal F}_{n}^{s}  \right)} \right\rceil$,
		we conclude that
		$\chi \left(  {\cal F}_{n}^{s}  \right) \geq \left\lceil   \frac{n}{a}   \right\rceil$.
		
		\noindent
		In order to prove
		$\chi \left(  {\cal F}_{n}^{s}  \right) = \left\lceil   \frac{n}{a}   \right\rceil$,
		it suffices to provide a proper vertex coloring of $ {\cal F}_{n}^{s} $
		with $\left\lceil   \frac{n}{a}   \right\rceil$ colors.
		In this regard, let us arrange all vertices $1, 2, 3, \dots , n$ increasingly from left to right in a straight line.
		Now, consider $C:= \left\{\gamma _{1} , \gamma _{2} , \dots , \gamma _{\left\lceil   \frac{n}{a}   \right\rceil} \right\}$ as a set of $\left\lceil   \frac{n}{a}   \right\rceil$ pairwise distinct colors. Also, define ${\cal L}$ as a block, which is the sequence of all colors $\gamma _{1} , \gamma _{2} , \dots , \gamma _{\left\lceil   \frac{n}{a}   \right\rceil}.$ More precisely,
		$${\cal L} := \left[\begin{array}{ccccccccc}
			\gamma _{1} &  ,   &   \gamma _{2} &  ,   &    \dots &   ,   &   \gamma _{\left\lfloor \frac{n}{a} \right\rfloor}  &  ,   &   
			\gamma _{\left\lceil   \frac{n}{a}   \right\rceil}
		\end{array}\right] .$$
		Similarly, define ${\cal S}$ as a block of the first
		$\left\lfloor \frac{n}{a} \right\rfloor$ colors, as follows :
		$${\cal S} := \left[\begin{array}{ccccccc}
			\gamma _{1} &  ,   &   \gamma _{2} &  ,   &    \dots &   ,   &   \gamma _{\left\lfloor \frac{n}{a} \right\rfloor}  
		\end{array}\right] .$$
		Now, an arrangement of $  n-a {\left\lfloor \frac{n}{a} \right\rfloor} $
		times the block ${\cal L}$, and then, $a {\left\lfloor \frac{n}{a} \right\rfloor} + a - n$
		times the block ${\cal S}$, that is, 
		$$\underbrace{{\cal L} \ , \  {\cal L} \ \dots \ , \ {\cal L}}_{ \left(  n-a \left\lfloor \frac{n}{a} \right\rfloor  \right) \ {\rm times}} \ ,
		\underbrace{{\cal S} \ , \  {\cal S} \ \dots \ , \ {\cal S}}_{ \left(  a {\left\lfloor \frac{n}{a} \right\rfloor} + a - n  \right) \ {\rm times}} $$
		provides a proper coloring of $ {\cal F}_{n}^{s} $ with
		$\left\lceil   \frac{n}{a}   \right\rceil$ colors; as required.
	}
\end{proof}

\noindent {\bf Proof of The Theorem \ref*{Forrgeqs1}.}	
        Due to The Lemma \ref*{restriction}, it suffices to show that
		$\left\lceil  \frac{{\rm cd}^r \left({\cal F}_{n}^{s}\right)}{r-1}   \right\rceil \leq 1.$ In this regard, we show that ${\rm cd}^r \left({\cal F}_{n}^{s}\right) \leq s-1 .$
		From the vertex set of ${\cal F}_{n}^{s}$ which is $[n]$, remove all vertices that are greater than or equal to $n-s+2$. The result is the induced subgraph of ${\cal F}_{n}^{s}$ on $[n-s+1]$. Now, color each vertex $x$ in $[n-s+1]$ by the unique $\mathring{x}$ in
		$\{0,1, \dots s-1\}$ for which $x$ and $\mathring{x}$ are congruent modulo $s$.
		
		\noindent
		Suppose that $x$ any $y$ are two vertices in $[n-s+1]$ in such a way that they are adjacent
		in ${\cal F}_{n}^{s}$. Since $|x-y|\leq |(n-s+1)-1|= n-s$, we must have $|x-y| < s$ (otherwise, we would have $s \leq |x-y|\leq n-s$; a contradiction to the fact that $x$ and $y$ are adjacent in ${\cal F}_{n}^{s}$).
		Thus, $x$ and $y$ are not congruent modulo $s$; and therefore, $\mathring{x} \neq \mathring{y}$. Hence, our remaining induced subgraph is $s$-colorable. Since $r\geq s$,
		it is also $r$-colorable, and we are done.
		
		\hfill{$\blacksquare$}


\noindent {\bf Proof of The Theorem \ref{directrs}.}
	Let $n:=3s-1$, and put ${\cal G} :=  {\cal F}_{n}^{s} $.
	We have
	$$\begin{array}[1]{ccc}
		n\geq 2r + 1   &   {\rm and}    &  \chi ({\rm KG}^r ({\cal G}_{s-{\rm stable}})) =0.
	\end{array}$$

    \noindent 
    Now, on account of The Theorem \ref{ChromaticNumber}, we obtain
    $$\chi ({\cal G})
     = \left\lceil  \frac{3s-1}{2} \right\rceil   > r.$$
     Therefore, ${\rm cd} ^r ({\cal G}) > 0$, and the assertion follows.
     
 	 \hfill{$\blacksquare$}


\noindent {\bf Proof of The Theorem \ref{inversers}.}
		Due to the hypotheses, The Lemma \ref{restriction} implies that
		$\left\lceil  \frac{{\rm cd}^r \left({\cal F}_{n}^{s}\right)}{r-1}   \right\rceil  \geq 1.$
		So, ${\rm cd}^r \left({\cal F}_{n}^{s}\right) > 0$; and therefore, $\chi \left({\cal F}_{n}^{s}\right) > r$.
		
		\noindent
		Since $r \geq s+1$, put $r=s+i$, where $i$ is a positive integer. Also, let $n=as+b$, where $a$ and $b$ are integers and $0 \leq b \leq s-1.$
		According to The Theorem \ref{ChromaticNumber}, we have
		$\chi \left({\cal F}_{n}^{s}\right) = \left\lceil  \frac{n}{a}  \right\rceil 
		= s + \left\lceil  \frac{b}{a}  \right\rceil.$
		Therefore,
		$$s + \left\lceil  \frac{b}{a}  \right\rceil = \chi \left({\cal F}_{n}^{s}\right) 
		> r=s+i;	$$
		which yields $b \geq ai+1.$
		Now, since $s-1 \geq b$ and $a \geq 2$, we conclude that $i \leq \frac{s-2}{2}$; which is equivalent to $i \leq \left\lceil  \frac{s-3}{2}  \right\rceil$; as required.

        \hfill{$\blacksquare$}
        
        
\begin{rem}
	If positive integers $r$ and $s$ satisfy $s\geq 4$ and $s+1 \leq r \leq \left\lceil  \frac{3s-3}{2}  \right\rceil$, one may ask for the finiteness of the set of all positive integers $n$ for which
	there exists a hypergraph over the ground set $[n]$ with
	$$\begin{array}{cccccc}
		n\geq 2r    &   ,    &  \chi ({\rm KG}^r (\left({\cal F}_{n}^{s}\right)_{s-{\rm stable}}))=0     &   ,      &
		{\rm and}    &    \left\lceil  \frac{{\rm cd}^r \left({\cal F}_{n}^{s}\right)}{r-1}   \right\rceil  = 1.   \\
	\end{array}$$
	We should note that this set is finite. Put $n=as+b$, where $0 \leq b \leq s-1$. Now, if
	$a \geq s-1$, then 
	$$\chi \left({\cal F}_{n}^{s}\right) = \left\lceil  \frac{n}{a}  \right\rceil 
	= s + \left\lceil  \frac{b}{a}  \right\rceil \leq s+1 \leq r;$$
	and therefore, ${\rm cd}^r \left({\cal F}_{n}^{s}\right) =0.$
	So, $n$ cannot be greater than or equal to $s(s-1)$. Consequently, the mentioned set is finite.    
\end{rem}

\noindent {\bf Proof of The Theorem \ref{Freers1}.}
	Let us regard $k:= l +1$ and $n:=sk+1$.
	Now, put ${\cal G} := {\cal F}_{n}^{s}$.		
	Since $E({\cal G}_{s-{\rm stable}}) = \varnothing$, we have
	$\chi ({\rm KG}^r ({\cal G}_{s-{\rm stable}}))=0$.
	Now, according to The Theorem \ref{ChromaticNumber}, it follows that
	$$\chi ({\cal G}) = \chi \left({\cal F}_{n}^{s}\right) = \left\lceil  \frac{n}{k}  \right\rceil 
	= s+1 > s;$$
	and therefore, ${\rm cd}^s \left(   {\cal G}  \right) > 0$ and
	$\left\lceil  \frac{{\rm cd}^s ({\cal G})}{r-1}   \right\rceil  \geq 1$; as desired.

    \hfill{$\blacksquare$}
    

\noindent {\bf Proof of The Theorem \ref{Freers2}.}
    We put $n:=ls(rs-1)$ and
	$A:= \{ls, 2ls, 3ls, \dots , (rs-2)ls , (rs-1)ls\}$. Also, let ${A \choose s}$ denote the set of all $s$-subsets of $A$, that is, ${A \choose s} := \{B : B \subseteq A \ {\rm and} \ |B|=s \}$.
	Now, consider the hypergraph ${\cal F}$ with vertex set $[n]$ and hyperedge set
	$E({\cal F}) := {A \choose s} \cup \left[   {[n] \choose 2} \setminus {[n] \choose 2}_{s-{\rm stable}}   \right]$.
	Since $E({\cal F}_{s-{\rm stable}}) = {A \choose s} \neq \varnothing$ and $|A| < rs$, we obtain
	$\chi ({\rm KG}^r ({\cal F}_{s-{\rm stable}}))=1$.
	Now, in order to prove
	$\left\lceil  \frac{{\rm cd}^s ({\cal F})}{r-1}   \right\rceil  \geq 2$,
	it suffices to show that
	${\rm cd}^s ({\cal F}) \geq r$.
	Suppose, on the contrary, that ${\rm cd}^s ({\cal F}) \leq r-1$.
	Hence, there exist $r-1$ vertices $x_{1} , x_{2} , \dots , x_{r-1}$ with
	$x_{1} < x_{2} < \dots < x_{r-1}$ in such a way that removing them from ${\cal F}$
	yields a remaining induced $s$-colorable hypergraph, which we call it $\tilde{{\cal F}}$.
	Also, let $c:V \left( \tilde{{\cal F}} \right) \longrightarrow C$ be a proper vertex coloring of $\tilde{{\cal F}}$, where $C$ is a set of colors whose size is $s$.
	
	\noindent 
	We define some circular intervals as follows.
	For each $i$ in $[r-2]$ define
	$$\begin{array}[3]{lll}
		(x_{i} , x_{i+1})   &  :=    &  \{ k \in {\mathbb{Z}} : x_{i} < k < x_{i+1}\};  \\
		                    &        &            \\
		(x_{i} , x_{i+1})_{ls} &  := &  \{ k \in (x_{i} , x_{i+1}) : ls \ {\rm divides} \ k\}.    \\
	\end{array}$$
    Also, define
	$$\begin{array}[3]{lll}
		(x_{r-1} , x_{1})   &  :=    &  \{ k \in {\mathbb{Z}} : x_{r-1} < k \leq n \ {\rm or} \ 1\leq k < x_{1}\};  \\
		&        &            \\
		(x_{r-1} , x_{1})_{ls} &  := &  \{ k \in (x_{r-1} , x_{1}) : ls \ {\rm divides} \ k\}.    \\
	\end{array}$$
	Now, there exist at least
	$(rs-1)-(r-1)$ vertices included in
	$$\left[   \bigcup _{i=1}^{r-2}   (x_{i} , x_{i+1})_{ls}    \right]   \bigcup   \  (x_{r-1} , x_{1})_{ls}.$$ Therefore, one of the $r-1$ pairwise disjoint sets
	$(x_{1} , x_{2})_{ls}$ or $(x_{2} , x_{3})_{ls}$ or $\dots $ or $(x_{r-2} , x_{r-1})_{ls}$
	or $(x_{r-1} , x_{1})_{ls}$ would contain at least
	$\left\lceil  \frac{r(s-1)}{r-1}   \right\rceil$ vertices; which certainly all of them
	receive a unit color with respect to $c:V \left(  \tilde{{\cal F}}  \right) \longrightarrow C$. We conclude the existence of $s$ vertices in $A \cap V \left(  \tilde{{\cal F}}  \right)$ with the same color; a contradiction to the $s$-colorability of $\tilde{{\cal F}}$. Accordingly, we have
	${\rm cd}^s ({\cal F}) \geq r$; and the assertion follows.	

    \hfill{$\blacksquare$}
    

\noindent {\bf Proof of The Theorem \ref{AlmosttoProblem3}.}
    Let $\chi ({\rm KG}^r ({\cal G}_{s-{\rm stable}})) = t.$
    So, there exists a proper vertex $t$-coloring
    $c: E({\cal G}_{s-{\rm stable}}) \longrightarrow \{\gamma_{1}, \gamma_{2}, \dots , \gamma_{t}\}$, where $\gamma_{1}, \gamma_{2}, \dots , \gamma_{t}$ are $t$ pairwise distinct colors.
    Now, we consider a color $\gamma_{t+1}$ with $\gamma_{t+1} \notin \{\gamma_{1}, \gamma_{2}, \dots , \gamma_{t}\}$
    and extend the coloring $c$ to a map
    $\tilde{c} : E \left( {\cal G}_{     \widetilde{{s-  {\rm stable}   }}     } \right)  \longrightarrow \{\gamma_{1}, \gamma_{2}, \dots , \gamma_{t} , \gamma_{t+1}\}$ by assigning the color $\gamma_{t+1}$ to each element of
    $E \left( {\cal G}_{     \widetilde{{s-  {\rm stable}   }}     } \right)
    - E ({\cal G}_{s-{\rm stable}}) $.
    
    \noindent Let
    $e \in E \left( {\cal G}_{     \widetilde{{s-  {\rm stable}   }}     } \right)
    \setminus E ({\cal G}_{s-{\rm stable}}) .$ So, $e \cap \{1 , 2 , \dots , s-1\} \neq \varnothing.$ Therefore, due to $s-1 < r$, if
    $\tilde{c} (e_1) = \tilde{c} (e_2) = \dots = \tilde{c} (e_{r-1}) = \tilde{c} (e_{r}) = \gamma_{t+1}$, then $e_{1} , e_{2} , \dots , e_{r-1} , e_{r}$ are not pairwise
    disjoint. It follows that
    $\tilde{c} : E \left( {\cal G}_{     \widetilde{{s-  {\rm stable}   }}     } \right)  \longrightarrow \{\gamma_{1}, \gamma_{2}, \dots , \gamma_{t} , \gamma_{t+1}\}$
    is a proper vertex coloring of
    ${\rm KG}^r \left( {\cal G}_{     \widetilde{{s-  {\rm stable}   }}     } \right)$;
    and consequently,
    $$\chi \left(       {\rm KG}^r \left( {\cal G}_{     \widetilde{{s-  {\rm stable}   }}     } \right)         \right)      -  \chi ({\rm KG}^r ({\cal G}_{s-{\rm stable}})) \leq 1.    
    $$
    Now, if The Conjecture \ref{Almost} is true, then we have    
    $$\left\lceil  \frac{{\rm cd}^r ({\cal G})}{r-1}   \right\rceil      -  \chi ({\rm KG}^r ({\cal G}_{s-{\rm stable}})) \leq    
    \chi \left(       {\rm KG}^r \left( {\cal G}_{     \widetilde{{s-  {\rm stable}   }}     } \right)         \right)      -  \chi ({\rm KG}^r ({\cal G}_{s-{\rm stable}})) \leq 1;    
    $$
    as claimed.

\hfill{$\blacksquare$}


\begin{thm}
	Let $r$ and $l$ be arbitrary positive integers, and $r\geq 2$.
	Then, there exists a graph ${\cal F}$ whose vertex set is $[lr(2r-1)]$, and satisfies
	$$\begin{array}[1]{ccc}
		\chi ({\rm KG}^r ({\cal F}_{r-{\rm stable}}))=1  &  {\rm and}   &
		\left\lceil  \frac{{\rm cd}^r ({\cal F})}{r-1}   \right\rceil  \geq 2.
	\end{array}  $$
\end{thm}

\begin{proof}{
	Put $n:=lr(2r-1)$, and consider the graph ${\cal F}$ with vertex set $[n]$
	and edge set
	$$E({\cal F}):= \bigg[   {[n] \choose 2}  \setminus {[n] \choose 2}_{r- {\rm stable} } \bigg] \bigcup \bigg\{  \{alr, blr\} : \{a,b\} \subseteq [2r-1] \ {\rm and} \ a \neq b   \bigg\}.$$
	Now,  ${\cal F}_{r-{\rm stable}}$ is a graph with vertex set $[n]$
	and edge set
	$$E({\cal F}_{r-{\rm stable}})= \bigg\{  \{alr, blr\} : \{a,b\} \subseteq [2r-1] \ {\rm and} \ a \neq b   \bigg\}.$$
	Since $E({\cal F}_{r-{\rm stable}}) \neq \varnothing$, and there are not $r$ pairwise disjoint edges in $E({\cal F}_{r-{\rm stable}})$, we find that
	$\chi ({\rm KG}^r ({\cal F}_{r-{\rm stable}}))=1$.
	
	\noindent The graph ${\cal F}$ has a clique $\{lr, 2lr, \dots , (2r-1)lr\}$ of size $2r-1$. So,
	at least $r-1$ vertices of this clique must be removed in order to remain an induced subgraph of
	${\cal F}$ which is $r$-colorable.
	So, ${\rm cd}^r ({\cal F}) \geq r-1$.
	We claim that ${\rm cd}^r ({\cal F}) \neq r-1$.
	Suppose, on the contrary, that ${\rm cd}^r ({\cal F}) = r-1$.
	Therefore, there exist some integers $a_{1} , a_{2} , \dots , a_{r-1}$
	with $1 \leq a_{1} < a_{2} < \dots < a_{r-1} \leq 2r-1$ in such a way that
	removing all vertices $a_{1}lr , a_{2}lr , \dots , a_{r-1}lr$
	from ${\cal F}$, yields an induced $r$-colorable subgraph $\tilde{{\cal F}}$, which
	admits a proper vertex coloring $c:V \left( \tilde{{\cal F}} \right) \longrightarrow C$,
	where $C$ is a set of $r$ colors.
	
	\noindent Now, for each $i$ in $[2r-2]$ define
	$$W_{i} := \{ ilr + j : 0\leq j \leq lr\}.$$
	Also, define
	$$W_{2r-1} := \{(2r-1)lr , 1, 2 , \dots , lr\}.$$
	For any two integers $i$ and $j$ with $1\leq i < j \leq 2r-1$, we have
	$$W_{i} \cap W_{j} = \left\{
	\begin{array}[3]{clll}
		\{jlr\}        &      &   {\rm if}            &   j=i+1,     \\
		\{lr\}         &      &   {\rm if}            &   \{i,j\}=\{1,2r-1\},    \\
		\varnothing    &      &   {\rm otherwise.}    &       
	\end{array}\right.$$
    Since each of $a_{1}lr , a_{2}lr , \dots , a_{r-1}lr$ is included in exactly two of
    $W_{1} , W_{2}, \dots , W_{2r-1}$,
    we conclude that there exists at least one $i$ in $[2r-1]$ for which
    $$W_{i} \cap \{a_{1}lr , a_{2}lr , \dots , a_{r-1}lr\} = \varnothing .$$
    Therefore, after removing $a_{1}lr , a_{2}lr , \dots , a_{r-1}lr$ from ${\cal F}$,
    all vertices of $W_{i}$ are still remained; or in other words, $W_{i} \subseteq V \left(  \tilde{{\cal F}}  \right)$.
    Without loss of generality, we may assume that $W_{i} = W_{1}$.
    
    \noindent The set $\{lr , lr+1 , lr+2 , \dots , lr+r-1\}$ represents a clique of size
    $r$ in $\tilde{{\cal F}}$. So, all colors in $C$ must appear on 
    $\{lr , lr+1 , lr+2 , \dots , lr+r-1\}$; that is, 
    $$\{c(lr) , c(lr+1) , \dots , c(lr+r-1)\} = C.$$
    Also, the set $\{lr+1 , lr+2 , \dots , lr+r-1 , lr+r\}$ represents a clique of size
    $r$ in $\tilde{{\cal F}}$; and therefore, 
    $$\{c(lr+1) , c(lr+2) , \dots , c(lr+r)\} = C.$$
    Consequently, we have $c(lr) = c(lr+r)$.
    Continuing in this fashion, we obtain
    $$c(lr) = c(lr+r) = c(lr+2r) = \dots = c(lr+lr) = c(2lr).$$
    So, on one hand, we have $c(lr) = c(2lr)$.
    On the other hand, we have $\{lr, 2lr\} \in E \left(  \tilde{\cal F}   \right)$.
    This is a contradiction to the properness of $c: V \left(  \tilde{{\cal F}} \right) \longrightarrow C$.
    Hence, ${\rm cd}^r ({\cal F}) \geq r$; and therefore, $\left\lceil \frac{{\rm cd}^r ({\cal F})}{r-1} \right\rceil  \geq 2$;
    as claimed.    
}
\end{proof}

\noindent {\bf Acknowledgments.} The author wishes to express his thanks to Hamid Reza Daneshpajouh for many stimulating conversations.

\def\cprime{$'$} \def\cprime{$'$}

\end{document}